\documentclass[leqno,12pt]{amsart} 

\usepackage{amssymb} 
\usepackage{amsmath}
\usepackage{amsthm}
\usepackage{amscd}
\usepackage{bm}
\usepackage{graphicx,psfrag,wrapfig}
\theoremstyle{plain} 
\newtheorem{theorem}{Theorem}[section] 
\newtheorem{lemma}[theorem]{Lemma}

\newtheorem{proposition}[theorem]{Proposition}

\newtheorem{conjecture}[theorem]{Conjecture}
\theoremstyle{definition} 

\newtheorem{remark}[theorem]{Remark}

\newcommand{\affine}{\mathbb{C}}

\newcommand{\norm}[1]{\left|\!\left|#1\right|\!\right|}
\newcommand{\supnorm}[1]{\left|\!\left|#1\right|\!\right|_{\infty}}
\newcommand{\moduli}{\mathcal{M}}

\newcommand{\Dnorm}[3]{\left|\!\left| #1 \right|\!\right|_{\mathcal{C}^{#2}(#3)}}

\newcommand{\Widim}{\mathrm{Widim}}
\newcommand{\Diam}{\mathrm{Diam}}

\newcommand{\vol}{\mathrm{vol}}
\newcommand{\dbar}{\bar{\partial}}
\newcommand{\const}{\mathrm{const}}

\begin{document}

\title[Deformation of Brody curves]{Deformation of Brody curves and mean dimension} 

\author[Masaki Tsukamoto]{Masaki Tsukamoto$^*$} 



\thanks{$^*$Supported by Grant-in-Aid for JSPS Fellows (19$\cdot$1530) from Japan Society for the
Promotion of Science}

\date{\today}




\maketitle

\begin{abstract}
The main purpose of this paper is to show that ideas of deformation theory can be 
applied to ``infinite dimensional geometry".
We develop the deformation theory of Brody curves.
Brody curve is a kind of holomorphic map from the complex plane to the 
projective space.
Since the complex plane is not compact, the parameter space of the deformation 
can be infinite dimensional.
As an application we prove a lower bound on the mean dimension of the space of 
Brody curves.
\end{abstract}

\section{Introduction}
\subsection{Main results}
Let $z=x+y\sqrt{-1}$ be the natural coordinate in the complex plane $\affine$.
For a holomorphic curve $f=[f_0:f_1:\cdots:f_N]:\affine \to \affine P^N$
with holomorphic functions $f_0, f_1: \cdots, f_N$, we define the pointwise norm 
$|df|\geq 0$ (with respect to the Fubini-Study metric) by 
\begin{equation}\label{definition of the norm of the differential}
 |df|^2 = \frac{1}{4 \pi} \Delta \log \left( |f_0|^2 +|f_1|^2+ \cdots +|f_N|^2 \right) \quad 
(\Delta := \frac{\partial^2}{\partial x^2} + \frac{\partial^2}{\partial y^2} ). 
\end{equation}
We call $f$ a Brody curve if it satisfies $|df|\leq 1$ (cf. Brody \cite{Brody}).
Let $\moduli(\affine P^N)$ be the space of Brody curves in $\affine P^N$ with 
the compact-open topology.
Then $\moduli(\affine P^N)$ becomes an infinite dimensional compact space and 
it admits a natural $\affine$-action:
\begin{equation}\label{definition of C-action}
 (f(z), a) \mapsto f(z+a) \quad \text{for a Brody curve $f(z)$ and $a\in \affine$}.
\end{equation}

This paper studies the ``mean dimension" $\dim(\moduli(\affine P^N):\affine)$.
Mean dimension is a notion defined by Gromov \cite{Gromov}
(see also Lindenstrauss-Weiss \cite{Lindenstrauss-Weiss} and Lindenstrauss \cite{Lindenstrauss}).
Mean dimension is a ``dimension of an infinite dimensional space".
Intuitively (the precise definition will be given in Section 2),
\[ ``\dim(\moduli(\affine P^N):\affine) = \dim \moduli(\affine P^N) /\vol(\affine)". \]
When we study the space of holomorphic maps from a compact Riemann surface, 
its (virtual) dimension can be derived from the deformation theory (and the index theorem).
The main purpose of this paper is to develop a new deformation theory which can be 
applied to the computation of $\dim(\moduli(\affine P^N):\affine)$.

For a Brody curve $f$ we define the Shimizu-Ahlfors characteristic function $T(r, f)$ by 
\[ T(r,f) := \int_{1}^{r}\frac{dt}{t}\int_{|z|<t} |df|^2 dxdy \leq \pi r^2/2. \]
We define the ``mean energy" $e(f)$ (see Tsukamoto \cite{T2}) by 
\[ e(f) := \limsup_{r\to \infty} \frac{2}{\pi r^2}T(r, f) \in [0,1] .\]
Let $e(\affine P^N)$ be the supremum of $e(f)$ over $f\in \moduli(\affine P^N)$.
From the definition we have $e(\affine P^N)\in [0,1]$, but actually we can prove 
(see Tsukamoto \cite{T1, T2})
\[ 0 < e(\affine P^N) <1 .\]
We call $f\in \moduli(\affine P^N)$ an elliptic Brody curve if there exists a
lattice $\Lambda\subset \affine$ such that $f(z+\lambda) = f(z)$ for all $z\in \affine$ and 
$\lambda\in \Lambda$. If $f$ is a non-constant elliptic Brody curve, then $e(f)>0$.
Let $e(\affine P^N)_{ell}$ be the supremum of $e(f)$ 
over elliptic Brody curves $f$ in $\affine P^N$.
Obviously $0< e(\affine P^N)_{ell}\leq e(\affine P^N)$.
Using the argument in Tsukamoto \cite[Section 4]{T1}, we can prove that 
$e(\affine P^N)_{ell}$ and $e(\affine P^N)$ asymptotically become equal to $1$:
\begin{equation}\label{eq: asymptotics of mean energy}
 \lim_{N\to \infty}e(\affine P^N)_{ell} = \lim_{N\to \infty} e(\affine P^N) = 1 .
\end{equation}
Our main result on the mean dimension is the following inequality:
\begin{theorem}\label{main theorem}
 \[2e(\affine P^N)_{ell} \,(N+1) \leq \dim(\moduli(\affine P^N):\affine) \leq 4 e(\affine P^N)N. \]
\end{theorem}
This theorem has the following two consequences:
\begin{theorem}\label{theorem in N=1}
\[ 4e(\affine P^1)_{ell}\leq \dim(\moduli(\affine P^1):\affine) \leq 4e(\affine P^1).\]
\end{theorem}
\begin{theorem}\label{theorem in N goes to infty}
\[ 2\leq \liminf_{N\to \infty}\, \dim(\moduli(\affine P^N):\affine))/N \leq 
\limsup_{N\to \infty}\, \dim(\moduli(\affine P^N):\affine)/N \leq 4 .\]
\end{theorem}
Theorem \ref{theorem in N=1} is the special case of Theorem \ref{main theorem}. 
Theorem \ref{theorem in N goes to infty} comes from (\ref{eq: asymptotics of mean energy}).
The point of Theorem \ref{theorem in N goes to infty} 
is that the estimate is explicit.
(The mean dimension $\dim(\moduli(\affine P^N):\affine)$ is a very transcendental object.)

Theorem \ref{theorem in N=1} leads us to the following conjecture
(actually a second main purpose of this paper is to propose this conjecture
to the mathematical community):
\begin{conjecture}\label{main conjecture}
\[ e(\affine P^1)_{ell} = e(\affine P^1).\]
\end{conjecture}
If this is true, then we get the following (index-theorem-like) result: 
\begin{equation}\label{mean dimension formula}
 \dim(\moduli(\affine P^1):\affine) = 4e(\affine P^1) .
\end{equation}
I think this formula is (if it is true) astonishing because the definitions of 
the left-hand-side and right-hand-side of (\ref{mean dimension formula}) are very different.
(Mean dimension is a topological quantity of the space, and mean energy is defined by 
using the energy distribution of Brody curves.)
Note that Conjecture \ref{main conjecture} itself is a purely function-theoretic problem.
It does not contain a notion in the mean dimension theory.

The upper bound, $\dim(\moduli(\affine P^N):\affine) \leq 4e(\affine P^N)N$, in Theorem 
\ref{main theorem} is already proved in Tsukamoto \cite[Theorem 1.4 and 1.5]{T2} 
by using the Nevanlinna theory\footnote{For the upper bound, 
see also Gromov \cite[p. 396, (c)]{Gromov} and Tsukamoto \cite[Remark 1.6]{T2}.}.
The task of this paper is to prove the lower bound: 
$\dim(\moduli(\affine P^N):\affine) \geq 2e(\affine P^N)_{ell}(N+1)$.
In order to prove this, we will develop a deformation theory of Brody curves.
This deformation theory is a step toward the ``infinite dimensional geometry":
The parameter space of the deformation can be infinite dimensional.
(But this is very natural because the space of Brody curves is an infinite dimensional 
space.)

A technical new feature of our deformation theory is the following:
Usually we construct deformation theory within the framework of ``$L^2$-theory"
(or sometimes $L^p$-theory for $p<\infty$).
But (I think that) $L^2$-theory is not suitable for our purpose and it is better to 
construct the theory in the settings of ``$L^\infty$-theory".
(The fact that $L^\infty$ is suitable for the mean dimension theory is also 
suggested by \cite{T-banach}.
In \cite{T-banach} it is shown that the mean dimension of the unit ball in 
$\ell^p(\Gamma)$ is zero, where $1\leq p<\infty$ and $\Gamma$ is a 
finitely generated infinite amenable group.) 
But the analysis in the $L^\infty$-settings is more complicated 
than that of $L^2$, and it is the main technical task of the paper.


\subsection{Remark on Conjecture \ref{main conjecture}}
An elliptic function $f$ constructed below might be a good candidate for 
the function which attains the supremum of $e(f)$.
Actually the following $f$ is an
extremal function of the Bloch-constant-type problem solved in Bonk-Eremenko \cite{Bonk-Eremenko}.
Put
\[ e_1 := 1/\sqrt{2}, \quad e_2 :=e^{2\pi \sqrt{-1}/3}/\sqrt{2}, \quad 
e_3 := e^{4\pi \sqrt{-1}/3}/\sqrt{2}, \quad
e_4 := \infty .\]
These four points become the vertices of a regular tetrahedron inscribed in 
the Riemann sphere $S^2 = \affine P^1$.
Let $\omega_1$ be a positive real number (which will be fixed later)
 and set $\omega_2 := \omega_1 \exp(\pi\sqrt{-1}/3)$.
Let $\Delta\subset \affine $ be the regular triangle whose vertices are $0$, $\omega_1$, $\omega_2$
and, $\tilde{\Delta}\subset \affine P^1$ the spherical regular triangle whose vertices are 
$e_1$, $e_2$, $e_4$.
From the Riemann mapping theorem there exists an (unique) one-to-one holomorphic map 
$f:\Delta \to \tilde{\Delta}$ which sends $0$, $\omega_1$, $\omega_2$ to $e_1$, $e_4$,
$e_2$ respectively.
From the reflection principle, 
$f$ can be extended to an elliptic function whose period lattice is 
$\Lambda := \mathbb{Z}(2\omega_1) \oplus \mathbb{Z}(2\omega_2)\subset \affine$.
The set of critical points of $f$ is $\mathbb{Z}\omega_1 + \mathbb{Z}\omega_2\subset \affine$,
and the critical values are $e_1$, $e_2$, $e_3$, $e_4$. 
We have $\deg(f:\affine/\Lambda \to \affine P^1) = 2$.
$f$ satisfies 
\[ (f')^2 = K (f-e_1)(f-e_2)(f-e_3) = K (f^3 - 1/\sqrt{8}) \]
for some positive constant $K$. 
$\omega_1$ can be derived from $K$ by
\[ \omega_1 = \frac{1}{\sqrt{K}}\int_{1/\sqrt{2}}^{\infty}\frac{dx}{\sqrt{x^3 -1/\sqrt{8}}}
= \frac{2^{1/4}}{\sqrt{K}}\int_1^\infty \frac{dx}{\sqrt{x^3-1}} .\]
The spherical derivative $|df|(z)$ defined in (\ref{definition of the norm of the differential})
 is given by 
\[ |df|^2 = \frac{1}{\pi}\frac{|f'|^2}{(1+ |f|^2)^2} =
\frac{K}{\pi}\frac{|f^3-1/\sqrt{8}|}{(1+ |f|^2)^2} .\]
Some calculation shows
\[ \sup_{z\in\affine}\frac{|z^3-1/\sqrt{8}|}{(1+|z|^2)^2} = 1/\sqrt{8} .\]
Therefore
\[ \sup_{z\in \affine}|df|^2(z) =\frac{K}{\pi\sqrt{8}} .\]
We choose $\omega_1$ so that $K = \pi\sqrt{8}$. 
Then $\sup_{z\in \affine}|df|(z) = 1$ and $f$ becomes an elliptic Brody curve.
Since the volume of the fundamental domain of $\Lambda$ in $\affine$ is given by 
$|\affine/\Lambda| = 2\sqrt{3}\omega_1^2$, we have
\[ e(f) = \frac{2}{|\affine /\Lambda|}  = \frac{2\pi}{\sqrt{3}} 
\left( \int_1^\infty \frac{dx}{\sqrt{x^3-1}}\right)^{-2}
= 0.6150198678198\cdots . \]
From Theorem \ref{theorem in N=1},
\[ \dim(\moduli(\affine P^1):\affine) \geq \frac{8\pi}{\sqrt{3}}
\left( \int_1^\infty \frac{dx}{\sqrt{x^3-1}}\right)^{-2}
=2.460079471279\cdots . \]
This inequality might be an equality.


\subsection{Remark on residual dimension}
We want to remark about the ``residual dimension" introduced by 
Gromov (see \cite[p. 330 and p. 346]{Gromov}).
This subsection is logically independent of the proof of Theorem \ref{main theorem},
and readers can skip it.
(But the idea of this subsection is implicitly used in Section 3.)
Let $\Lambda\subset \affine$ be a lattice and 
$\moduli(\affine P^N)_\Lambda$ be the set of Brody curves $f$ satisfying 
$f(z+\lambda) = f(z)$ for all $\lambda\in \Lambda$.
$\moduli(\affine P^N)_\Lambda$ is the set of fixed-points of the natural action 
of $\Lambda$ on $\moduli(\affine P^N)$.
In other words, $\moduli(\affine P^N)_\Lambda$ is the space of holomorphic maps 
$f:\affine /\Lambda\to \affine P^N$ satisfying $|df|\leq 1$.
The usual deformation theory gives (cf. Section 3)
\[ \frac{1}{|\affine/\Lambda|}\dim \moduli(\affine P^N)_\Lambda \leq 
2(N+1)\sup_{f\in \moduli(\affine P^N)_\Lambda} e(f) \leq 2e(\affine P^N)_{ell}(N+1) .\]
In particular, Theorem \ref{main theorem} gives 
\begin{equation}\label{eq: resdim < meandim}
 \begin{split}
 \mathrm{resdim}(\moduli(\affine P^N): \{n\Lambda\}_{n\geq 1}) &:= 
 \liminf_{n\to \infty}\frac{1}{|\affine/n\Lambda|}\dim \moduli(\affine P^N)_{n\Lambda} ,\\
 &\leq  2e(\affine P^N)_{ell}(N+1) \leq \dim(\moduli (\affine P^N): \affine) .
 \end{split}
\end{equation}
Moreover some consideration shows 
\[ \sup_{\Lambda\subset \affine}\mathrm{resdim}(\moduli(\affine P^N) : \{n\Lambda\}_{n\geq 1})
= 2e(\affine P^N)_{ell} (N+1) ,\]
where $\Lambda$ runs over all lattices in $\affine$.
\begin{remark}
In (\ref{eq: resdim < meandim}) the residual dimension is not bigger than the mean dimension.
But in general residual dimension can be bigger than mean dimension;
Consider the natural action of $\mathbb{Z}$ on $[0,1]^{\mathbb{Z}}$.
For $n\geq 1$ let $F_n\subset [0,1/n]^{\mathbb{Z}}$ be the set of fixed-points
of the action of $n\mathbb{Z}$ on $[0,1/n]^{\mathbb{Z}}$.
Set $X:= \bigcup_{n\geq 1}F_n$.
$X$ becomes a $\mathbb{Z}$-invariant closed set in $[0,1]^{\mathbb{Z}}$.
Let $X_n$ $(n\geq 1)$ be the set of fixed-points
of the action of $n\mathbb{Z}$ on $X$.
Since $F_n\subset X_n$, we have $\dim X_n = n$. 
Therefore, 
$\mathrm{resdim}(X:\{n\mathbb{Z}\}_{n\geq 1}) := \liminf_{n\to \infty} \dim X_n/n = 1$.
On the other hand, it is not difficult to see $\dim(X:\mathbb{Z}) = 0$.
\end{remark}

\subsection{Remark: twisted-elliptic Brody curves}
For a Brody curve $f:\affine \to \affine P^N$, we call $f$ 
a twisted-elliptic Brody curve if there exist a lattice $\Lambda\subset \affine$ and 
a homomorphism (of groups) $\phi: \Lambda\to PU(N+1)$ such that 
\[ f(z +\lambda) = \phi(\lambda) f(z) \quad \text{for all $z\in \affine$ and $\lambda\in\Lambda$} .\]
Note that the projective unitary group $PU(N+1)$ is the holomorphic-isometry group of $\affine P^N$.
Perhaps it might be able to apply the methods in this paper to twisted-elliptic Brody curves also.
I think this is a natural generalization.
But I don't know whether this improves the estimate of the mean dimension or not.
So I don't study this case in this paper.
If there is a reader who has an interest in this case, please pursue it.

\subsection{Organization of the paper}
In Section 2 we review the definition and basic properties of mean dimension.
In Section 3 we prove Theorem \ref{main theorem}, assuming 
an analytic result about the ``deformation theory of Brody curves" proved in Section 5.
Section 4 is a preparation for Section 5. 
In Section 5 we develop the deformation theory of Brody curves and complete 
the proof of Theorem \ref{main theorem}.
We give a remark about 
Gromov's conjecture on rational curves and mean dimension in Section 6.

\subsection{Acknowledgement}
I wish to thank Professors Minoru Murata and Yoshio Tsutsumi.
They gave me several helpful advices on elliptic partial differential equations.
Especially I learned the basic idea of the proof of Proposition 4.2 from 
Professor Minoru Murata.
I also wish to thank Professor Katsutoshi Yamanoi for various valuable discussions.
He gave me an important suggestion about Conjecture \ref{main conjecture}.


\section{Review of mean dimension}
We review the definitions of mean dimension. For the detail, 
see Gromov \cite{Gromov} or Lindenstrauss-Weiss \cite{Lindenstrauss-Weiss}.
Let $(X, d)$ be a compact metric space, $Y$ a topological space.
For $\varepsilon >0$, a continuous map $f:X\to Y$ is called an 
$\varepsilon$-embedding if we have $\mathrm{Diam} f^{-1}(y) \leq \varepsilon$ for all 
$y\in Y$.
Let $\Widim_\varepsilon (X, d)$ be the minimum number $n\geq 0$ such that 
there are an $n$-dimensional polyhedron $K$ and an $\varepsilon$-embedding from 
$X$ to $K$.
The following is proved in Gromov \cite[p. 333]{Gromov}.
(This is a basic result for us. So we will give its proof in Appendix.)
\begin{proposition}\label{prop: widim inequality}
Let $(V, \norm{\cdot})$ be an $n$-dimensional normed linear space (over $\mathbb{R}$).
Let $B\subset V$ be the closed ball of radius $r>0$ with the distance $d(x, y) := \norm{x-y}$.
Then 
\[ \Widim_\varepsilon (B, d) = n \quad \text{for all $\varepsilon <r$} .\]
\end{proposition}
Suppose the Lie group $\affine$ continuously acts on the compact metric space $X$.
For any positive number $R$, we define the distance $d_R(\cdot, \cdot)$ on $X$ by 
\[ d_R(p, q) := \sup_{z\in \affine,\, |z|\leq R}d(z.p, z.q) \quad \text{for $p, q\in X$} .\]
Set 
\[ \Widim_\varepsilon (X:\affine) := 
\lim_{R\to \infty} \frac{1}{\pi R^2} \Widim_\varepsilon (X, d_R) .\]
This limit always exists (see Gromov \cite[pp. 335-338]{Gromov} and Lindenstrauss-Weiss 
\cite[Appendix]{Lindenstrauss-Weiss}).
We define the mean dimension $\dim(X:\affine )$ by setting 
\[ \dim(X:\affine):= \lim_{\varepsilon \to 0}\Widim_\varepsilon (X: \affine) .\]
$\dim(X:\affine)$ is a topological invariant, i.e., it does not depend on the given distance $d$.

Let $\Lambda = \mathbb{Z}\omega_1 \oplus \mathbb{Z}\omega_2\subset \affine$ be a lattice
($\omega_1, \omega_2 \in \affine$).
Then $\Lambda$ also acts on $X$ and we can define the mean dimension 
$\dim(X:\Lambda)$ as follows: For any positive integer $n$ we set 
\begin{equation}\label{def: Omega_n}
 \Omega_n := \{ x\omega_1 + y\omega_2 \in \Lambda |\, x, y\in \mathbb{Z},\, 0\leq x, y\leq n-1 \}
\end{equation}
We define the distance $d_{\Omega_n}(\cdot, \cdot)$ on $X$ by 
\begin{equation}\label{def: d_Omega_n}
 d_{\Omega_n}(p, q):= \max_{z\in \Omega_n} d(z.p, z.q) \quad \text{for $p, q\in X$} .
\end{equation}
Set (the following limit always exists)
\[ \Widim_\varepsilon (X:\Lambda) := \lim_{n\to \infty}\frac{1}{n^2} 
\Widim_\varepsilon (X, d_{\Omega_n}).\]
We define the mean dimension $\dim(X:\Lambda)$ by 
\[ \dim(X:\Lambda) := \lim_{\varepsilon\to 0}\Widim_\varepsilon (X:\Lambda) .\]
The following gives the relation between $\dim(X:\affine)$ and $\dim(X:\Lambda)$. 
(This is given in Gromov \cite[p.329]{Gromov} and
 Lindenstrauss-Weiss \cite[Proposition 2.7]{Lindenstrauss-Weiss}. 
For its proof, see also Tsukamoto \cite[Proposition 4.5]{T2}.)
\begin{proposition}\label{prop: relation between lambda and affine}
\[ \dim(X:\Lambda) = |\affine/\Lambda| \dim(X:\affine) ,\]
where $|\affine/\Lambda|$ denotes the volume the fundamental domain of $\Lambda$ in $\affine$.
\end{proposition}


\section{Proof of Theorem \ref{main theorem}}\label{Section: Proof of the main Theorem}
Let $\Lambda\subset \affine$ be a lattice and $\pi:\affine \to \affine /\Lambda$ be 
the natural projection.
Let $\varphi :\affine/\Lambda\to \affine P^N$ be a non-constant holomorphic map 
satisfying $|d\varphi| <1$, and
set $\tilde{\varphi}:= \varphi\circ\pi :\affine \to \affine P^N$.
We have 
\begin{equation*}
 e(\tilde{\varphi}) = \deg\varphi /|\affine /\Lambda | ,
\end{equation*}
where $\deg \varphi = \langle c_1(\varphi^*\mathcal{O}(1)), [\affine /\Lambda]\rangle$.
Let $T'\affine P^N$ be the holomorphic tangent bundle of $\affine P^N$ and 
consider its pull-back $E:= \tilde{\varphi}^* T'\affine P^N$ over $\affine$.
$E$ is equipped with the Hermitian metric induced by the Fubini-Study metric.
We define a Banach space $V$ as the space of bounded holomorphic sections of $E$
with the sup-norm $\supnorm{\cdot}$:
\begin{equation}\label{definition of the tangent space}
 V:=\{u:\affine \to E|\, \text{$u$ is a holomorphic section and satisfies 
$\supnorm{u}:=\sup_{z\in \affine}|u(z)| <\infty$}\}.
\end{equation}
The following result is the keystone of the proof of Theorem \ref{main theorem}.
\begin{proposition}\label{prop: main proposition}
There are positive numbers $\delta$ and $C$ such that 
for any $u\in V$ with $\supnorm{u}\leq \delta$ there exists a Brody curve
$f_u :\affine \to \affine P^n$ satisfying the following:

(i) $f_0 = \tilde{\varphi}$.

(ii) The map $B_\delta\ni u\mapsto f_u\in \moduli(\affine P^N)$ is $\Lambda$-equivariant.
Here $B_\delta =\{u\in V| \supnorm{u}\leq \delta\}$ and
we have considered the natural $\Lambda$-action on $E$ and $V$.

(iii) For any $u, v\in V$ with $\supnorm{u}, \supnorm{v}\leq \delta$, we have
\[ C^{-1}\supnorm{u-v} \leq \sup_{z\in \affine} \underline{d}(f_u(z), f_v(z)) \leq 
C \supnorm{u-v} ,\]
where $\underline{d}(\cdot, \cdot)$ denotes the distance on 
$\affine P^N$ defined by the Fubini-Study metric.
\end{proposition}
We will prove this proposition in Section 5 by constructing a ``deformation theory" of $\tilde{\varphi}$.
(Each $f_u$ is a ``small deformation" of $\tilde{\varphi}$.)
Here we prove Theorem \ref{main theorem}, assuming Proposition \ref{prop: main proposition}.
\begin{proof}[Proof of Theorem \ref{main theorem}]
To begin with, we define the distance $d(\cdot, \cdot)$ on $\moduli(\affine P^N)$ by
\[ d(f, g) := \sum_{n \geq 1} 2^{-n}\sup_{|z|\leq n} \underline{d}(f(z), g(z)) \quad 
\text{for $f, g\in \moduli(\affine P^N)$} .\]
Let $\Lambda = \mathbb{Z} \omega_1 \oplus \mathbb{Z}\omega_2\subset \affine$ be a lattice
in $\affine$ ($\omega_1, \omega_2 \in \affine$). For any positive integer $n$ we set 
\begin{equation}\label{eq: definition of K_n}
 K_n := \{x\omega_1 + y\omega_2 \in \affine |\, x,y\in \mathbb{R},\, 0\leq x, y\leq n\}.
\end{equation}
$K_n$ is a fundamental domain of $n\Lambda$ in $\affine$.
There is a positive constant $C_1 = C_1(\Lambda)$ such that 
\[ \sup_{z\in K_1}\, \underline{d}(f(z), g(z)) \leq C_1 \, d(f, g) \quad 
\text{for $f, g\in \moduli(\affine P^N)$}.\]
Then for any $n>0$ we have
\begin{equation}\label{eq: sup < d}
\sup_{z\in K_n} \underline{d}(f(z), g(z)) \leq C_1 \, d_{\Omega_n}(f, g) \quad
\text{for $f, g\in \moduli(\affine P^N)$},
\end{equation}
where $\Omega_n$ and $d_{\Omega_n}(\cdot, \cdot)$ are defined by (\ref{def: Omega_n}) 
and (\ref{def: d_Omega_n}).

Let $\varphi :\affine /\Lambda \to \affine P^N$ be a non-constant holomorphic map satisfying 
$|d\varphi|<1$.
We define $\tilde{\varphi}$, $E$ and $V$ as before. 
For any positive integer $n$, let $\pi_n:\affine/n\Lambda \to \affine /\Lambda$ be the 
natural $n^2$-fold covering map, and set 
$\varphi_n:= \varphi \circ \pi_n :\affine /n\Lambda \to \affine P^N$.
Consider 
\[ V_n := H^0(\affine/n\Lambda, \mathcal{O}(\varphi_n^*T'\affine P^N)) .\]
$V_n$ is the space of holomorphic sections of $\varphi_n^*T'\affine P^N$ over 
$\affine /n\Lambda$, and it can be identified with the subspace of $V$ consisting of
$n\Lambda$-invariant holomorphic sections of $E$.
From the Riemann-Roch formula and 
the vanishing of $H^1$ (cf. Section 5), we have
\[ \dim V_n = 2\dim_\affine V_n 
= 2\langle \varphi_n^*c_1 (\affine P^N), [\affine /n\Lambda]\rangle
= 2n^2(N+1) \deg\varphi .\]
(Actually we need only the inequality $\dim V_n \geq  2n^2(N+1) \deg\varphi$ in this proof.
Hence we don't need $H^1 = 0$.)
Let $\delta, C$ be the positive constants in Proposition \ref{prop: main proposition}.
Set $B_\delta(V_n) := \{ u\in V_n| \supnorm{u}\leq \delta\}$.
For any $u\in B_\delta(V_n)$ there exists a Brody curve $f_u$. 
From the $\Lambda$-equivariance in Proposition \ref{prop: main proposition} (ii),
$f_u$ is $n\Lambda$-invariant (i.e., it can be considered as a holomorphic map 
from $\affine /n\Lambda$ to $\affine P^N$).
Then from Proposition \ref{prop: main proposition} (iii) and (\ref{eq: sup < d}),
for any $u, v\in B_\delta(V_n)$
\[ \supnorm{u-v}\leq C\sup_{z\in \affine}\underline{d}(f_u(z), f_v(z)) 
= C\sup_{z\in K_n}\underline{d}(f_u(z), f_v(z)) \leq CC_1\, d_{\Omega_n}(f_u, f_v) .\]
Moreover Proposition \ref{prop: main proposition} shows that the map 
$B_\delta(V_n) \to \moduli(\affine P^N)$, $u\mapsto f_u$, is continuous.
Therefore for any $\varepsilon >0$
\[ \Widim_\varepsilon (\moduli ( \affine P^N), d_{\Omega_n} ) 
\geq \Widim_{CC_1\varepsilon}(B_\delta(V_n), \supnorm{\cdot}), \]
where $B_\delta(V_n)$ is equipped with the distance $\supnorm{u-v}$.
Then Proposition \ref{prop: widim inequality} implies, for $\varepsilon < \delta/CC_1$,
\[ \Widim_\varepsilon (\moduli (\affine P^N), d_{\Omega_n}) \geq \dim V_n = 
2n^2(N+1) \deg\varphi .\]
Note that $\delta/CC_1$ is independent of $n$ (this is the crucial point). Hence
\[ \Widim_\varepsilon(\moduli(\affine P^N):\Lambda) = 
\lim_{n\to \infty} \frac{1}{n^2}\Widim(\moduli(\affine P^N), d_{\Omega_n}) 
\geq 2(N+1)\deg\varphi  ,\]
for any $\varepsilon < \delta/CC_1$. Thus
\[ \dim(\moduli(\affine P^N):\Lambda) \geq 2(N+1)\deg\varphi .\]
Using Proposition \ref{prop: relation between lambda and affine}, we get
\begin{equation}\label{eq: conclusion of the proof of main theorem}
 \begin{split}
 \dim(\moduli(\affine P^N):\affine) &=
 \frac{1}{|\affine/\Lambda|}\dim(\moduli(\affine P^N):\Lambda) ,\\
 &\geq 2(N+1)\deg\varphi/|\affine /\Lambda| = 2(N+1) e(\tilde{\varphi}) .
 \end{split}
\end{equation}

Then we can prove Theorem \ref{main theorem}.
Let $f\in \moduli(\affine P^N)$ be any elliptic Brody curve.
Take a positive number $c<1$ and set $g(z):= f(cz)$. Then 
$g$ is an elliptic Brody curve satisfying $|dg|<1$, and we can apply 
(\ref{eq: conclusion of the proof of main theorem}) to $g$:
\[ \dim(\moduli(\affine P^N):\affine) \geq 2(N+1)e(g) = 2c^2 (N+1)e(f) .\]
Let $c\to 1$. Then
\[ \dim(\moduli(\affine P^N):\affine) \geq 2(N+1)e(f) .\]
This shows Theorem \ref{main theorem}.
\end{proof}
\begin{remark}
In the above proof, each $B_\delta(V_n)$ describes a small deformation of 
$\varphi_n :\affine/n\Lambda \to \affine P^N$.
The small deformations of each $\varphi_n$ can be constructed by the usual deformation theory.
The point of Proposition \ref{prop: main proposition} is that 
we can construct the deformations of all $\varphi_n$ with the estimates
\textit{independent of} $n$; This is essential in the above proof.
\end{remark}


\section{Analytic preliminaries}
This section is a preparation for the proof of Proposition \ref{prop: main proposition}.
\subsection{Helmholtz equation}
We will need some elementary facts about the Helmholtz equation on the plane $\mathbb{R}^2$:
\begin{equation}\label{eq: Helmholtz equation}
 (-\Delta + \lambda) w = 0, \quad \text{where $\lambda > 0$ and 
$\Delta =  \frac{\partial^2}{\partial x^2} + \frac{\partial^2}{\partial y^2}$} .
\end{equation}
Set 
\begin{equation}\label{eq: special solution of Helmholtz equation}
 w_\lambda(z) := \frac{1}{2\pi}\int_0^{2\pi}\exp \sqrt{\lambda}(x\cos \theta + y\sin \theta )d\theta .
\end{equation}
$w_\lambda$ satisfies (\ref{eq: Helmholtz equation}) and $w_\lambda >0$.
The following fact can be easily checked:
\begin{lemma} \label{lemma: property of Helmholtz equation}
The minimum value of $w_\lambda$ is $w_\lambda(0) = 1$, and 
$w_\lambda(z) \to + \infty$ as $|z| \to \infty$.
\end{lemma}
\subsection{$L^\infty$-estimate}
Let $F$ be a holomorphic vector bundle over the complex plane $\affine$
with a Hermitian metric $h$.
Let $\dbar :\Omega^0(F) \to \Omega^{0,1}(F)$ be the Dolbeault operator, and
$\nabla$ the canonical connection on $(F,h)$.
We denote the formal adjoint of $\dbar$ and $\nabla$ by $\dbar^*$ and $\nabla^*$.
We have the following Weintzenb\"{o}ck formula: for any $\xi\in \Omega^{0,1}(F)$
\begin{equation}\label{Weintzenbock formula}
 \dbar \dbar^* \xi = \frac{1}{2}\nabla^*\nabla \xi + R \xi , 
\end{equation}
where $R \xi = [\nabla_{\partial /\partial z}, \nabla_{\partial /\partial\bar{z}}]\xi$.
Note that for $\xi = u\otimes d\bar{z}$ ($u\in \Gamma(F)$) we have
\[ \nabla_{\partial /\partial z}\xi = (\nabla_{\partial /\partial z}u)\otimes d\bar{z}, \quad
 \nabla_{\partial /\partial\bar{z}}\xi = (\nabla_{\partial /\partial\bar{z}}u)\otimes d\bar{z}.\]
For $\xi =u\otimes d\bar{z}$ and $\eta = v\otimes d\bar{z}$ ($u, v\in \Gamma(F)$), we set
$\langle\xi , \eta \rangle := 2h(u, v)$.
We suppose that $F$ is ``positive" in the following sense:
there exists a positive number $a$ such that for any $\xi\in \Omega^{0,1}(F)$
\begin{equation}\label{eq: positivity of Hermitian metric}
\langle R \xi, \xi\rangle \geq a|\xi|^2 .
\end{equation}
\begin{proposition} \label{prop: L^infty-estimate}
Let $\xi\in \Omega^{0,1}(F)$ be a $F$-valued $(0,1)$-form of class $\mathcal{C}^2$, 
and set $\eta := \dbar\dbar^* \xi$. 
If $\supnorm{\xi}, \supnorm{\eta} <\infty$, then 
\[ \supnorm{\xi} \leq \frac{8}{a}\supnorm{\eta} .\]
\end{proposition}
\begin{proof}
There is a point $z_0\in \affine$ satisfying $|\xi(z_0)|\geq \supnorm{\xi}/2$.
We suppose $z_0 =0$ for simplicity. We have
\[ \Delta |\xi|^2 = -2\mathrm{Re}\langle \nabla^*\nabla\xi, \xi\rangle + 2|\nabla \xi|^2 .\]
Using the Weintzenb\"{o}ck formula (\ref{Weintzenbock formula}) and $\eta = \dbar\dbar^* \xi$, 
we have
\begin{equation*}
 \begin{split}
 \Delta |\xi|^2 &= -4 \mathrm{Re}\langle \eta, \xi\rangle + 4\langle R\xi , \xi\rangle
 + 2|\nabla\xi|^2, \\
 &\geq -4 \mathrm{Re}\langle \eta, \xi\rangle + 4a|\xi|^2 .
 \end{split}
\end{equation*}
Set $M:= 4\supnorm{\xi}\supnorm{\eta}$. We have
$(-\Delta + 4a)|\xi|^2 \leq M$.

Set $w(z) := M w_{2a}(z)/2a$, where $w_{2a}$ is a function defined in 
(\ref{eq: special solution of Helmholtz equation}).
$w(z)$ satisfies
\[ (-\Delta + 2a)w =0, \quad w \geq M/2a .\]
Then $(-\Delta + 4a)w = 2aw \geq M$. Therefore
\[ (-\Delta + 4a)(w-|\xi|^2) \geq 0 .\]
Since $\supnorm{\xi} <\infty$ and $w(z) \to \infty$ $(|z|\to \infty)$, we have
$w(z) -|\xi|^2 >0$ $(|z| \gg 0)$.
Then we can apply the minimum principle 
(see Gilbarg-Trudinger \cite[Chapter 3, Corollary 3.2]{Gilbarg-Trudinger}), and get 
\[ w(0) - |\xi(0)|^2 \geq 0 .\]
Therefore 
\[ \supnorm{\xi}^2/4 \leq |\xi(0)|^2 \leq w(0) = M/2a = 2\supnorm{\xi}\supnorm{\eta}/a .\]
Thus $\supnorm{\xi} \leq 8\supnorm{\eta}/a$.
\end{proof}

\subsection{Perturbation of a Hermitian metric}
We briefly discuss a perturbation technique of a Hermitian metric. 
M. Gromov also discuss it in \cite[p. 399]{Gromov}.
Let $\Lambda\subset \affine$ be a lattice and $\varphi:\affine/\Lambda \to \affine P^N$
a non-constant holomorphic map. 
Let $\varphi^*T'\affine P^N \to \affine/\Lambda$ be the pull-back of the holomorphic tangent bundle 
$T'\affine P^N$ with the Hermitian metric $h$ induced by the Fubini-Study metric.
Since the holomorphic bisectional curvature of the Fubini-Study metric is positive, there is $c>0$
such that for any $u\in \Gamma(\varphi^*T'\affine P^N)$
\begin{equation}\label{eq: positivity of the Fubini-Study metric}
 h(Ru, u) \geq c|d\varphi|^2 |u|^2, 
\end{equation}
where $R$ is the curvature defined by 
$Ru:= [\nabla_{\partial /\partial z}, \nabla_{\partial /\partial\bar{z}}]u$ 
($\nabla$ is the canonical connection).
\begin{lemma}\label{lemma: perturbation of a Hermitian metric}
There is a Hermitian metric $h'$ on $\varphi^*T'\affine P^N$ satisfying the following:
There exists $a>0$ such that for any $u\in \Gamma(\varphi^*T'\affine P^N)$
\[ h'(R'u, u) \geq a |u|^2 , \]
where $R'$ is the curvature of $h'$.
\end{lemma}
\begin{proof}
Set $h' = e^{-f}h$ where $f$ is a real valued function defined later.
Then for any $u \in \Gamma(\varphi^*T'\affine P^N)$
\[ R' u = \frac{1}{4}(\Delta f) u + Ru \quad\text{and}\quad
h'(R'u, u) = e^{-f}\{\frac{1}{4}(\Delta f) |u|^2 + h(Ru, u)\} .\]
Set $\{p\in \affine /\Lambda|\, d\varphi (p) =0\} =: \{p_1, \cdots, p_n\}$.
Let $\delta>0$ be a sufficiently small number and set 
$A:= \coprod_i B_\delta(p_i) \subset \affine /\Lambda$ 
($B_\delta(p_i)$ is the closed ball of radius $\delta$ centered at $p_i$).
From (\ref{eq: positivity of the Fubini-Study metric}) there is $c' >0$ such that 
\[ h(Ru, u) \geq c' |u|^2 \quad \text{for $u\in (\varphi^*T'\affine P^N)_p$ at 
$p\in A^c = (\affine /\Lambda)\setminus A$} .\]
Let $g$ be a real valued function on $\affine /\Lambda$ satisfying 
\[ (i)\, g > 0 \text{ on $A$}, \quad (ii)\, g \geq -c'/2 \text{ on $A^c$}, \quad
(iii)\, \int_{\affine /\Lambda} g \,dxdy = 0 .\]
From the condition (iii), there exists $f$ satisfying $\Delta f/4 =g$.
Therefore
\[ h'(R'u, u) = e^{-f}(g\, |u|^2 + h(Ru,u)) .\]
From the conditions (i) and (ii),
it is easy to see that there exists $a>0$ such that $h'(R'u, u) \geq a|u|^2$
for all sections $u$.
\end{proof}


\section{Deformation theory}
In this section we prove Proposition \ref{prop: main proposition}
by constructing ``deformation theory".
\begin{remark}
M. Gromov gives a certain ``deformation" argument different from ours
in \cite[pp. 399-400]{Gromov}.
\end{remark}
\subsection{Deformation and the proof of Proposition \ref{prop: main proposition}}
Let $\Lambda\subset \affine$ be a lattice and 
$\pi:\affine \to \affine /\Lambda$ the natural projection.
Let $\varphi:\affine /\Lambda\to \affine P^N$ be a non-constant 
holomorphic map satisfying $|d\varphi|<1$ and set $\tilde{\varphi} := \varphi\circ\pi$.
Let $E:= \tilde{\varphi}^*T'\affine P^N$ be the pull-back of the holomorphic tangent bundle
$T'\affine P^N$. $E$ is equipped with the Hermitian metric $h$ induced by the Fubini-Study metric.
$E$ admits the natural $\Lambda$-action.

Let $k$ be a non-negative integer and $\alpha$ a real number satisfying $0<\alpha<1$.
We want to define the H\"{o}lder spaces $\mathcal{C}^{k, \alpha}(E)$ and 
$\mathcal{C}^{k,\alpha}(\Omega^{0,1}(E))$.
Let $\{U_n\}_{n=1}^m$, $\{U'_n\}_{n=1}^m$ and $\{U''_n\}_{n=1}^m$
be open coverings of $\affine /\Lambda$ satisfying 
the following (i), (ii), (iii).

(i) $\bar{U}_n \subset U'_n$ and $\bar{U}'_n \subset U''_n$, and all 
$U_{n}$, $U'_{n}$, $U''_n$ are smooth regions i.e., their boundaries are smooth.

(ii) The covering map $\pi:\affine \to \affine /\Lambda$ can be trivialized on 
each $U''_n$, i.e., there is a disjoint union 
$\pi^{-1}(U''_n) = \coprod_{\lambda\in \Lambda} U''_{n,\lambda}$ such that 
each $U''_{n,\lambda}$ is a connected component of $\pi^{-1}(U''_n)$ and 
$\pi|_{U''_{n,\lambda}}:U''_{n,\lambda}\to U''_n$ is biholomorphic.
Set $U_{n,\lambda} := \pi^{-1}(U_n)\cap U''_{n,\lambda}$ and
$U'_{n,\lambda} := \pi^{-1}(U'_n)\cap U''_{n,\lambda}$, then 
$\pi|_{U_{n,\lambda}}:U_{n,\lambda}\to U_n$ and $\pi|_{U'_{n,\lambda}}:U'_{n,\lambda}\to U'_n$
are biholomorphic and
we have disjoint unions $\pi^{-1}(U_n) = \coprod_{\lambda\in \Lambda} U_{n,\lambda}$ and
$\pi^{-1}(U'_n) = \coprod_{\lambda\in \Lambda} U'_{n,\lambda}$.

(iii) A bundle trivialization of $\varphi^*T'\affine P^N$ is given on each $U''_n$,
i.e., we have a holomorphic bundle isomorphism 
$\varphi^*T'\affine P^N|_{U''_n}\to U''_n \times \affine^N$.
Then we also have a trivialization of $E$ over each $U''_{n,\lambda}$
through the isomorphisms $U''_{n,\lambda}\to U''_n$.

Let $u$ be a section of $E$ (not necessarily holomorphic). 
From (iii) in the above, $u|_{U''_{n,\lambda}}$ can be seen as a vector-valued function 
on $U''_{n,\lambda}$. Hence we can define its $\mathcal{C}^{k,\alpha}$-norm 
$\Dnorm{u}{k,\alpha}{\bar{U}_{n,\lambda}}$ over $\bar{U}_{n,\lambda}$
as a vector-valued function
(see Gilbarg-Trudinger \cite[Chapter 4]{Gilbarg-Trudinger}). 
We define the $\mathcal{C}^{k,\alpha}(E)$-norm of $u$ by 
\[ \Dnorm{u}{k,\alpha}{E} := \sup_{n,\lambda}\Dnorm{u}{k,\alpha}{\bar{U}_{n,\lambda}} .\]
We define the H\"{o}lder space $\mathcal{C}^{k,\alpha}(E)$ as the space of 
sections of $E$ whose $\mathcal{C}^{k,\alpha}(E)$-norms are finite.
For $\xi = u\otimes d\bar{z} \in \Omega^{0,1}(E)$ $(u\in \Gamma(E))$ we define its 
$\mathcal{C}^{k,\alpha}(\Omega^{0,1}(E))$-norm by
\[ \Dnorm{\xi}{k,\alpha}{\Omega^{0,1}(E)} := \sqrt{2} \Dnorm{u}{k,\alpha}{E} ,\]
and we define 
$\mathcal{C}^{k,\alpha}(\Omega^{0,1}(E)) := \mathcal{C}^{k,\alpha}(E)\otimes d\bar{z}$.
Then $\mathcal{C}^{k,\alpha}(E)$ and $\mathcal{C}^{k,\alpha}(\Omega^{0,1}(E))$ become Banach spaces.
(In the above definition of the H\"{o}lder spaces we have not used the open sets $U'_{n,\lambda}$.
They will be used in the next subsection.)

The holomorphic tangent bundle $T'\affine P^N$ is the eigenspace of the complex 
structure $J$ on $T\affine P^N \otimes_{\mathbb{R}} \affine$ of eigenvalue $\sqrt{-1}$.
We naturally identify $T'\affine P^N$ with the tangent bundle $T\affine P^N$ by
\[ T\affine P^N \ni u \longleftrightarrow u-\sqrt{-1}Ju \in T'\affine P^N .\]
So $E$ can be identified with $\tilde{\varphi}^*T\affine P^N$.

Consider (cf. McDuff-Salamon \cite[Chapter 3]{McDuff-Salamon})
\[ \Phi: \mathcal{C}^{1, \alpha}(E) \to \mathcal{C}^{0,\alpha}(\Omega^{0,1}(E)) \quad 
u\mapsto P_u(\dbar\exp u)\otimes d\bar{z} .\]
Here, $\exp :T\affine P^N \to \affine P^N$ is the exponential map defined by the Fubini-Study metric, 
and $P_u: T_{\exp u}\affine P^N \to T_{\tilde{\varphi}}\affine P^N$ is the parallel transport
along the geodesic $\exp tu$ $(0\leq t\leq 1)$.
$\dbar\exp u\in T_{\exp u}\affine P^N$ is defined by 
\[ \dbar\exp u := 
\frac{1}{2}\left(\frac{\partial}{\partial x}\exp u + 
J \frac{\partial}{\partial y}\exp u\right) .\]
$\Phi$ is a smooth map between the Banach spaces, and it is $\Lambda$-equivariant.
The map $\affine \ni z\mapsto \exp u(z)\in \affine P^N$ 
becomes a holomorphic curve if and only if $\Phi (u) =0$.
The derivative of $\Phi$ at the origin is the Dolbeault operator:
\begin{equation}\label{eq: the derivative at the origin}
 (d\Phi)_0 = \dbar :\mathcal{C}^{1, \alpha}(E) \to \mathcal{C}^{0,\alpha}(\Omega^{0,1}(E)).
\end{equation}
\begin{proposition}\label{prop: unobstructedness}
The small deformation of $\tilde{\varphi}$ is unobstructed, i.e.,  
there exists a $\Lambda$-equivariant bounded linear operator 
$Q:  \mathcal{C}^{0,\alpha}(\Omega^{0,1}(E)) \to \mathcal{C}^{1, \alpha}(E)$ satisfying 
$\dbar\circ Q = 1$.
\end{proposition}
This proposition will be proved later.
Let $V := \ker \dbar$ be the kernel of (\ref{eq: the derivative at the origin})
(this definition coincides with (\ref{definition of the tangent space})).
Note that $V$ is a complement of the image of $Q$ in $\mathcal{C}^{1,\alpha}(E)$
and that it is $\Lambda$-invariant.
From the elliptic regularity (cf. Subsection 5.2), we have 
\[ \supnorm{u} \leq  \mathrm{const}\cdot \Dnorm{u}{1,\alpha}{E} \leq \mathrm{const}'\cdot \supnorm{u}
\quad \text{for any $u\in V$}, \]
where $\mathrm{const}$ and $\mathrm{const}'$ are independent of $u$. 

For $r>0$ set $B_r:=\{u\in V| \supnorm{u}\leq r\}$.
From Proposition \ref{prop: unobstructedness} and the implicit function theorem,
there are $\delta >0$ and a $\Lambda$-equivariant smooth map $g:B_\delta \to \mathrm{Image} (Q)$
satisfying 
\begin{equation}\label{eq: conditions of g}
(i)\, g(0) = 0, \quad (ii)\, \Phi(u + g(u)) = 0 \text{ for all $u\in B_\delta$}, 
\quad (iii)\, (dg)_0 = 0 .
\end{equation}
Set $f_u := \exp(u+g(u)):\affine \to \affine P^N$ for $u\in B_\delta$.
We want to show that these $f_u$ satisfy the conditions in 
Proposition \ref{prop: main proposition}.
From (i) and (ii) in (\ref{eq: conditions of g}),
$f_0 = \tilde{\varphi}$ and each $f_u$ is a holomorphic curve.
Since $|d\varphi|<1$, if we choose $\delta$ sufficiently small, all $f_u$ $(u\in B_\delta)$ become
Brody curves, i.e., $|df_u|\leq 1$.
Since $g$ is $\Lambda$-equivariant, the map $B_\delta\ni u\mapsto f_u\in \moduli(\affine P^N)$
is also $\Lambda$-equivariant.

If we choose $\delta>0$ sufficiently small, then there exists $K >0$ such that 
for all $u, v\in B_\delta$
\[ K^{-1}\supnorm{u + g(u) -v-g(v)} \leq \sup_{z\in \affine}\underline{d}(f_u(z), f_v(z))
\leq K\supnorm{u + g(u) -v-g(v)} \]
(this is a standard property of the exponential map) and we have
\[ \supnorm{g(u)-g(v)} \leq \frac{1}{2}\supnorm{u-v} .\]
Here we have used the condition (iii) in (\ref{eq: conditions of g}).
Hence
\[ \frac{1}{2}K^{-1}\supnorm{u-v} \leq \sup_{z\in \affine}\underline{d}(f_u(z), f_v(z))
\leq \frac{3}{2}K\supnorm{u-v} .\]
Then all the conditions in Proposition \ref{prop: main proposition} have been proved
(assuming Proposition \ref{prop: unobstructedness}).


\subsection{Proof of Proposition \ref{prop: unobstructedness}}
To begin with, we consider a perturbation of the Hermitian metric on $E$.
$E$ has the Hermitian metric $h$ induced by the Fubini-Study metric.
From Proposition \ref{lemma: perturbation of a Hermitian metric}, 
$\varphi^*T'\affine P^N$ admits a Hermitian metric which is ``positive" in the sense of 
Proposition \ref{lemma: perturbation of a Hermitian metric}.
Then, pulling back this metric to $E$, $E$ admits a $\Lambda$-invariant Hermitian 
metric $h'$ satisfying (\ref{eq: positivity of Hermitian metric})
for some $a>0$.
In this subsection we use this $h'$ as the Hermitian metric on $E$.
Note that the definitions of the H\"{o}lder spaces 
$\mathcal{C}^{k,\alpha}(E)$ and $\mathcal{C}^{k,\alpha}(\Omega^{0,1}(E))$
does not use the Hermitian metric. So they are independent of the choice of the
Hermitian metric. 
(The sup-norm $\supnorm{\cdot}$ depends on the Hermitian metric, but the sup-norms defined by 
$h$ and $h'$ are equivalent to each other.)

We prove Proposition \ref{prop: unobstructedness} by showing that
\begin{equation}\label{eq: dbar-dbar^*}
\dbar\dbar^*:\mathcal{C}^{2,\alpha}(\Omega^{0,1}(E))\to \mathcal{C}^{0,\alpha}(\Omega^{0,1}(E))
\end{equation}
is an isomorphism.
(Note that the Dolbeault operator $\dbar$ is independent of the Hermitian metric $h'$, but 
its formal adjoint $\dbar^*$ depends on $h'$.)
Then $Q := \dbar^* (\dbar\dbar^*)^{-1}$ gives a $\Lambda$-equivariant right inverse of $\dbar$.
The injectivity of (\ref{eq: dbar-dbar^*}) directly follows from
the $L^\infty$-estimate in Proposition \ref{prop: L^infty-estimate}.
So the problem is its surjectivity.
\begin{lemma}\label{lemma: existence of xi for compact-supported eta}
If $\eta\in \mathcal{C}^{0,\alpha}(\Omega^{0,1}(E))$ has a compact support, then 
there exists $\xi\in \mathcal{C}^{2,\alpha}(\Omega^{0,1}(E))$ satisfying 
$\dbar\dbar^* \xi = \eta$.
\end{lemma}
\begin{proof}
We set 
\[ L^2_1(\Omega^{0,1}(E)):= \{\xi\in L^2(\Omega^{0,1}(E))|\, \nabla \xi \in L^2\},\]
where $\nabla \xi$ is the distributional derivative of $\xi$.
(The $L^2$-norm and the $L^2$-space are defined by using the Hermitian metric $h'$.)
Let $\xi\in \Omega^{0,1}(E)$ be a compact-supported smooth section.
From the Weitzenb\"{o}ck formula (\ref{Weintzenbock formula}),
\begin{equation*}
 \begin{split}
 \norm{\dbar^*\xi}_{L^2}^2 &= (\dbar\dbar^*\xi, \xi)_{L^2} = 
 (\frac{1}{2}\nabla^*\nabla\xi + R\xi, \xi)_{L^2},\\
 &\geq \frac{1}{2}\norm{\nabla\xi}_{L^2}^2 + a\norm{\xi}_{L^2}^2 .
 \end{split}
\end{equation*}
Therefore for any $\xi \in L^2_1(\Omega^{0,1}(E))$ 
\[ \norm{\dbar^*\xi}_{L^2}^2\geq \frac{1}{2}\norm{\nabla\xi}_{L^2}^2 + a\norm{\xi}_{L^2}^2 .\]
This means that the inner-product $(\dbar^*\xi_1, \dbar^*\xi_2)_{L^2}$ 
$(\xi_1,\xi_2\in L^2_1(\Omega^{0,1}(E)))$
is equivalent to the natural inner-product $(\xi_1, \xi_2)_{L^2_1}:=
(\nabla \xi_1, \nabla \xi_2)_{L^2} + (\xi_1, \xi_2)_{L^2}$ 
on $L^2_1(\Omega^{0,1}(E))$.

$\eta$ defines a bounded functional $(\cdot, \eta)_{L^2}: L^2_1(\Omega^{0,1}(E))\to \affine$.
From the Riesz representation theorem, there (uniquely) exists $\xi\in L^2_1(\Omega^{0,1}(E))$
satisfying $(\dbar^*\phi, \dbar^*\xi)_{L^2} = (\phi, \eta)_{L^2}$ for all 
$\phi\in L^2_1(\Omega^{0,1}(E))$ and
\begin{equation}\label{eq: L^2_1 estimate}
\norm{\xi}_{L^2_1}:= (\xi, \xi)_{L^2_1}^{1/2} \leq \mathrm{const}\cdot \norm{\eta}_{L^2} .
\end{equation}
In particular, $\dbar\dbar^* \xi = \eta$ in the sense of distribution.
(The above is a standard argument in the ``$L^2$-theory".)

We want to show $\xi\in \mathcal{C}^{2,\alpha}(\Omega^{0,1}(E))$.
Remember the open covering 
$\affine = \bigcup_{n,\lambda}U_{n,\lambda} = \bigcup_{n,\lambda} U'_{n,\lambda}
= \bigcup_{n,\lambda}U''_{n,\lambda}$
$(n=1,\cdots, m, \, \lambda\in \Lambda)$
used in the definition of the H\"{o}lder spaces.
Each $\xi|_{U''_{n,\lambda}}$ can be seen as a 
vector-valued function.
From the Sobolev embedding $L^2_2\hookrightarrow \mathcal{C}^0$,
the elliptic regularity (see Gilbarg-Trudinger \cite[Chapter 8]{Gilbarg-Trudinger}) and 
(\ref{eq: L^2_1 estimate}),
\begin{equation*}
 \begin{split}
 \supnorm{\xi|_{U_{n,\lambda}}} &\leq \const_n \cdot \norm{\xi|_{U_{n,\lambda}}}_{L^2_2}
 \leq \const'_n \left(\norm{\xi|_{U'_{n,\lambda}}}_{L^2_1} + 
 \norm{\eta|_{U'_{n,\lambda}}}_{L^2}\right) \\
 &\leq \const''_n \cdot \norm{\eta}_{L^2},
  \end{split}
\end{equation*}
where $\const_n$, $\const'_n$ and $\const''_n$ are positive constants which depend on 
$n = 1, \cdots, m$.
The point is that they are independent of $\lambda\in \Lambda$; 
this is due to the $\Lambda$-symmetry of the equation.
Then 
\[ \supnorm{\xi}\leq \const \cdot \norm{\eta}_{L^2}  .\]
From the Schauder interior estimate (see Gilbarg-Trudinger \cite[Chapter 6]{Gilbarg-Trudinger}),
\[ \Dnorm{\xi}{2,\alpha}{\bar{U}_{n,\lambda}} \leq 
\const_n \left( \supnorm{\xi} + \Dnorm{\eta}{0,\alpha}{\bar{U}'_{n,\lambda}} \right)  
\leq \const\, (\norm{\eta}_{L^2} + \Dnorm{\eta}{0,\alpha}{\Omega^{0,1}(E)})  .\]
Here we have used the following fact (which can be easily checked):
\begin{equation}\label{eq: change of the open sets}
 \sup_{n,\lambda} \Dnorm{\eta}{0,\alpha}{\bar{U}'_{n,\lambda}} \leq 
 \const \cdot \Dnorm{\eta}{0,\alpha}{\Omega^{0,1}(E)} \, 
 (= \const \cdot\sup_{n,\lambda} \Dnorm{\eta}{0,\alpha}{\bar{U}_{n,\lambda}})
\end{equation}
Thus $\Dnorm{\xi}{2,\alpha}{\Omega^{0,1}(E)} <\infty$ and 
$\xi \in \mathcal{C}^{2,\alpha}(\Omega^{0,1}(E))$.
\end{proof}
Then we can prove that (\ref{eq: dbar-dbar^*}) is surjective (and hence isomorphic).
Take an arbitrary $\eta\in \mathcal{C}^{0,\alpha}(\Omega^{0,1}(E))$.
Let $\phi_k$ $(k\geq 1)$ be cut-off functions on the plane $\affine$ such that 
$0\leq \phi_k\leq 1$, $\phi_k(z) =1 $ for $|z|\leq k$ and $\phi_k(z) = 0$ for $|z|\geq k+1$.
Set $\eta_k := \phi_k \eta$.
From Lemma \ref{lemma: existence of xi for compact-supported eta}, there exists
$\xi_k\in \mathcal{C}^{2,\alpha}(\Omega^{0,1}(E))$ satisfying 
$\dbar\dbar^* \xi_k = \eta_k$.
From the $L^\infty$-estimate in Proposition \ref{prop: L^infty-estimate},
\begin{equation}\label{eq: boundedness of l^infty norm}
\supnorm{\xi_k}\leq \const \cdot \supnorm{\eta_k} \leq \const\cdot \supnorm{\eta} .
\end{equation}
Using the Schauder interior estimate on each $U'_{n,\lambda}$, we get 
\[ \Dnorm{\xi_k}{2,\alpha}{\bar{U}_{n,\lambda}} 
\leq \const_n \left( \Dnorm{\eta_k}{0,\alpha}{\bar{U}'_{n,\lambda}} + \supnorm{\xi_k}\right)
\leq \const'_n \left( \Dnorm{\eta_k}{0,\alpha}{\bar{U}'_{n,\lambda}} + \supnorm{\eta}\right).\]
Since $\eta_k|_{U'_{n,\lambda}} =\eta|_{U'_{n,\lambda}}$ for $k\gg 1$ (for each fixed $(n,\lambda)$),
$\{\xi_k|_{U_{n,\lambda}}\}_{k\geq 1}$ is a bounded sequence in 
$\mathcal{C}^{2,\alpha}(\bar{U}_{n, \lambda})$.
Hence, if we choose a subsequence, $\{\xi_k|_{U_{n,\lambda}}\}_{k\geq 1}$ 
becomes a convergent sequence in $\mathcal{C}^2(\bar{U}_{n,\lambda})$ 
(by Arzela-Ascoli's theorem).
Therefore (by using the diagonal argument) 
there exists $\xi\in \Omega^{0,1}(E)$ of class $\mathcal{C}^2$ such that 
$\{\xi_k|_{U_{n,\lambda}}\}_{k\geq 1}$ converges to $\xi|_{U_{n,\lambda}}$ in $\mathcal{C}^2(\bar{U}_{n,\lambda})$ 
for each $(n, \lambda)$.
Since $\dbar\dbar^* \xi_k = \eta_k$, we have $\dbar\dbar^* \xi = \eta$.
From (\ref{eq: boundedness of l^infty norm}), 
$\supnorm{\xi} \leq \const\cdot \supnorm{\eta} <\infty$.
Using the Schauder interior estimate, we get 
\[ \Dnorm{\xi}{2,\alpha}{\bar{U}_{n,\lambda}}\leq 
\const_n \left(\Dnorm{\eta}{0,\alpha}{\bar{U}'_{n,\lambda}} + \supnorm{\xi}\right) 
\leq \const'_n \left(\Dnorm{\eta}{0,\alpha}{\bar{U}'_{n,\lambda}} + \supnorm{\eta}\right).\]
From (\ref{eq: change of the open sets}),
\[ \Dnorm{\xi}{2,\alpha}{\Omega^{0,1}(E)} \leq 
\const\cdot \Dnorm{\eta}{0,\alpha}{\Omega^{0,1}(E)} < \infty .\]
Therefore we get $\xi\in \mathcal{C}^{2,\alpha}(\Omega^{0,1}(E))$ satisfying 
$\dbar\dbar^*\xi = \eta$. 
Then (\ref{eq: dbar-dbar^*}) is an isomorphism, and the proof of 
Proposition \ref{prop: unobstructedness} is finished
(and hence the proof of Theorem \ref{main theorem} is completed).


\section{Remark on Gromov's conjecture on rational curves and mean dimension}
Gromov gives the following (very beautiful) conjecture in \cite[p. 329]{Gromov}.
\begin{conjecture}\label{Gromov's conjecture}
Let $X\subset \affine P^N$ be a projective manifold, and 
$\moduli(X)$ the space of Brody curves in $X$.
Then $\dim(\moduli(X):\affine) >0$ if and only if 
$X$ contains a rational curve.
\end{conjecture}
The ``if" part is easy and the problem is the ``only if" part.
The purpose of this section is to show the following proposition:
\begin{proposition}\label{prop: counter-example}
There exists a compact Hermitian manifold $X$ such that 
$X$ contains no rational curve and satisfies 
$\dim(\moduli(X):\affine) >0$.
Here $\moduli(X)$ is the space of holomorphic maps $f:\affine \to X$ satisfying 
\[ \sup_{z\in \affine}|df|(z) := \sup_{z\in \affine}\sqrt{2}|df(\partial/\partial z)| \leq 1 .\]
\end{proposition}
This shows that the projectivity (or the K\"{a}hler condition) is essential 
in Conjecture \ref{Gromov's conjecture}.
(Actually I feel that the following argument suggests that the true 
conjecture might be something like the following; if $\dim(\moduli(X):\affine)>0$ then there are 
``many" elliptic curves in $X$ (cf. Gromov \cite[p. 330, EXAMPLE]{Gromov}).)
We prove Proposition \ref{prop: counter-example} by using an argument similar to that of  
Section \ref{Section: Proof of the main Theorem}.
(But this case is much easier than the proof of Theorem \ref{main theorem};
We don't need a serious analytic argument. 
In particular we don't use the results in Section 4,5.
Perhaps we can also prove Proposition \ref{prop: counter-example} by applying 
the argument in Gromov \cite[pp. 385-388]{Gromov} to the following construction.)

The compact Hermitian manifold $X$ constructed below is actually
known as a counter-example 
of ``bend-and-break" technique for general complex manifolds 
(see Koll\'{a}r-Mori \cite[Example 1.8]{Kollar-Mori}).
We follow the description of \cite[Example 1.8]{Kollar-Mori}.

Let $\affine/\Lambda$ be an elliptic curve
($\Lambda$ is a lattice in $\affine$).
Let $L$ be a holomorphic line bundle of $\deg \geq 2$ over $\affine/\Lambda$ such that 
there exists two holomorphic sections $s, t$ of $L$ satisfying 
$\{z\in \affine/\Lambda|\, s(z) = t(z) = 0\} =\emptyset$.
Set $F := L\oplus L$. The vector bundle $F$ has the following four sections:
\[ (s, t), \quad (\sqrt{-1}s, -\sqrt{-1}t), \quad (t, -s), \quad (\sqrt{-1}t, \sqrt{-1}s) .\]
These are $\mathbb{R}$-linearly independent all over $\affine/\Lambda$.
(Therefore $F$ becomes a product bundle as a real vector bundle.)
Hence we can define a lattice bundle $\Gamma \subset F$ by
\[ \Gamma := \{ x_1 (s, t) + x_2 (\sqrt{-1}s, -\sqrt{-1}t) +
 x_3 (t, -s) + x_4 (\sqrt{-1}t, \sqrt{-1}s) |\, x_1, x_2, x_3, x_4 \in \mathbb{Z}\} .\]
We define a compact complex threefold $X$ by $X := F/\Gamma$.
(Topologically $X= T^2 \times T^4 = T^6$.)
Obviously $X$ contains no rational curve.
But $X$ can contain lots of Brody curves as we will see below.

We give a Hermitian metric (of a complex vector bundle) to $F$ and a Hermitian 
metric (of a complex manifold) to $X$.
Let $\pi :\affine \to \affine /\Lambda$ be the natural projection and
$E:=\pi^* F$ the pull-back of $F$ by $\pi$.
$E$ is equipped with the $\Lambda$-invariant Hermitian metric induced by the metric on $F$.
Let $V$ be the (Banach) space of bounded holomorphic sections of $E$ with 
the sup-norm $\supnorm{\cdot}$, and set $B_\delta(V) := \{u\in V| \supnorm{u}\leq \delta\}$
for $\delta >0$.
Let $p:F\to X$ be the natural projection.
If we choose $\delta$ sufficiently small and consider some scale-change of the 
Hermitian metric of $X$, then, for any $u\in B_\delta(V)$, 
$p\circ u:\affine \to X$ belongs to $\moduli(X)$ and the map 
$\Phi: B_\delta(V) \ni u \mapsto p\circ u\in \moduli(X)$ becomes injective.
(Here we consider $u\in B_\delta(V)$ as a map from $\affine$ to $F$.)
We define a distance $d(\cdot, \cdot)$ on $B_\delta(V)$ by 
\[ d(u, v) := \sum_{n \geq 1} 2^{-n} \sup_{|z|\leq n} |u(z)-v(z)| 
\quad \text{for any $u, v\in B_\delta(V)$} .\]
We consider the topology defined by this distance on $B_\delta(V)$.
Then $B_\delta(V)$ becomes compact, and $\Phi:B_\delta(V)\to \moduli(X)$ becomes a
$\Lambda$-equivariant continuous embedding (here we consider the compact-open 
topology on $\moduli(X)$).
Hence $\dim(\moduli(X): \Lambda) \geq \dim(B_\delta(V):\Lambda)$.
Let $\pi_n:\affine/n\Lambda \to \affine /\Lambda$ be the natural 
$n^2$-fold covering.
Then the argument in Section 3 shows
\[ \dim(B_\delta(V) :\Lambda) \geq \lim_{n\to \infty} \frac{1}{n^2}
\dim H^{0}(\affine/n\Lambda :\mathcal{O}(\pi_n^*F)) = 2\deg(F) >0.\]
(This is an inequality of the type ``mean-dimension $\geq$ residual-dimension".)
Therefore 
\[ \dim(\moduli(X) :\affine) = \dim(\moduli(X):\Lambda)/|\affine/\Lambda| >0 .\]

\begin{remark}
The above $X$ does not admit a K\"{a}hler metric.
In fact the space of holomorphic one-forms in $X$ is (complex) one-dimensional.
Since the first Betti number of $X =T^6$ is $6$, the Hodge theory implies that 
there is no K\"{a}hler metric on $X$.
\end{remark}


\appendix
\section{Proof of Proposition \ref{prop: widim inequality}}
Gromov \cite[p. 333]{Gromov} proved Proposition \ref{prop: widim inequality} by using 
the notion ``filling radius". (Filling radius is a notion introduced in his celebrated 
paper \cite{Gromov-filling}.)
Our following proof is a variant of the argument of Lindenstrauss-Weiss 
\cite[Lemma 3.2]{Lindenstrauss-Weiss}.

It is enough to prove that for the unit ball 
$B:=\{x\in V|\norm{x}\leq 1\}$ we have 
\[ \Widim_\varepsilon (B,d) = \dim V = n \quad \text{for $\varepsilon <1$} .\]
Suppose there exists $\varepsilon <1$ such that $\Widim_\varepsilon(B,d) \leq n-1$.
Then there is a finite open covering $\{U_i\}_{i\in I}$ of $B$ such that 
$\delta:= \max_{i\in I} \Diam\, U_i <1$ and its order is $\leq n-1$, i.e., 
$U_{i_1}\cap U_{i_2}\cap \cdots\cap U_{i_{n+1}} = \emptyset$ for distinct 
$i_1, i_2, \cdots, i_{n+1} \in I$.

Let $\{\phi_i\}_{i\in I}$ be a partition of unity on $B$ satisfying 
$\mathrm{supp}\, \phi_i \subset U_i$. 
Take an arbitrary point $p_i$ in $U_i$.
We define a map $f:B\to B$ by $f(x):= -\sum_{i\in I}\phi_i(x)\cdot p_i$.
For any $x\in B$ we have 
\begin{equation}\label{eq: a certain norm inequality}
 \norm{f(x) + x} = \norm{\sum_{i\in I}\phi_i(x)(x-p_i)} \leq \delta\sum_{i\in I}\phi_i = \delta.
\end{equation}

For each $x\in B$ we have $\sharp\{i\in I|\phi_i(x) \neq 0\}\leq n$.
Therefore $f(B)$ is contained in a union of at most $n-1$ dimensional polyhedrons.
In particular $f(B)$ does not contain an inner point.
Hence there exists $a\in B$ such that $a\notin f(B)$ and $\norm{a} \leq 1-\delta$.
Then we can define $g:B\to \partial B$ by $g(x):= (f(x)-a)/\norm{f(x)-a}$.
$g$ does not have a fixed point. In fact if $g(x) = x$, then $x\in \partial B$ and
$f(x)-a = x\norm{f(x)-a}$. Then $f(x) + x -a = x(1+\norm{f(x) -a})$ and
$\norm{f(x) + x-a} = 1 + \norm{f(x)-a} >1$. 
From (\ref{eq: a certain norm inequality}),
\[ 1 < \norm{f(x)+x-a} \leq \norm{f(x) +x} + \norm{a} \leq \delta + \norm{a} \leq 1.\]
This is a contradiction.
Therefore $g$ does not have a fixed point, and this contradicts the Brouwer 
fixed-point theorem.

\vspace{10mm}

\address{ Masaki Tsukamoto \endgraf
Department of Mathematics, Faculty of Science \endgraf
Kyoto University \endgraf
Kyoto 606-8502 \endgraf
Japan
}

\textit{E-mail address}: \texttt{tukamoto@math.kyoto-u.ac.jp}

\end{document}